%% file: arxiv_submission.tex
\documentclass[12pt, oneside]{amsart}

\input{preamble.tex}
\usepackage[margin = 1in]{geometry}

\title[Fluctuation of the giant via BFW]{Fluctuations of the giant of Poisson random graphs}
\author{David Clancy, Jr.}
 \thanks{\noindent University of Wisconsin, Department of Mathematics. Email: dclancy@math.wisc.edu}
\date{\today}
\usepackage{verbatim}
\usepackage{todonotes}
\begin{document}

\maketitle

\begin{abstract}
  Enriquez, Faraud, and Lemaire (2023) have established process-level fluctuations for the giant of the dynamic Erd\H{o}s-R\'{e}nyi random graph above criticality and show that the limit is a centered Gaussian process with continuous sample paths. A random walk proof was recently obtained by Corujo, Limic and Lemaire (2024). We show that a similar result holds for rank-one inhomogeneous models whenever the empirical weight distribution converges to a limit and its second moment converges as well.
\end{abstract}



\section{Introduction}

Consider the Erd\H{o}s-R\'enyi random graph $(\G_n^\ER(\lambda);\lambda\ge 0)$ where each of the $\binom{n}{2}$ many edges is included independently with probability $1-e^{-\lambda/n}$. We will think of this graph as dynamic, with each edge $e = \{i,j\}$, $i<j$, appearing in $\G^\ER_n(\lambda)$ if and only if $E_{i,j}\le \lambda/n$ where $E_{i,j}\sim \Exp(1)$ are a family of independent exponential random variables. It is well-known \cite{ER.60} that the largest connected component $L_n^\ER(\lambda)$ undergoes a phase transition at $\lambda = 1$ where $L_n^\ER(\lambda) = \Theta_\PR(n)$ when $\lambda>1$ while for $\lambda<1$ the $L_n^\ER(\lambda) =O_\PR(\log(n))$. Moreover, above criticality $(\lambda>1)$ it holds that
\begin{equation*}
n^{-1} L_n^\ER(\lambda) \weakarrow \varrho^{\textup{ER}}(\lambda)
\end{equation*}
where $\varrho^{\textup{ER}}(\lambda)$ is the unique solution $x\in(0,1)$ to $1-e^{-\lambda x} = x$. In \cite{Stepanov.70}, Stepanov established a central limit theorem for the fluctuations of $L_n^\ER(\lambda)$ around $\varrho^{\textup{ER}}(\lambda) n$. In a formulation of Pittel \cite{Pittel.90}, for each fixed $\lambda:$
\begin{align}\label{eqn:CLT}
    \frac{L_n^\ER(\lambda) - \varrho^{\textup{ER}}(\lambda)n}{n^{1/2}}\weakarrow \mathscr{N}\left(0,\sigma^2(\lambda)\right) \quad\textup{where}\quad  \sigma^2(\lambda) = \frac{\varrho^{\textup{ER}}(\lambda)(1-\varrho^{\textup{ER}}(\lambda))}{(1-\lambda (1-\varrho^{\textup{ER}}(\lambda)))^2}
\end{align}
where $\mathscr{N}(0,\sigma^2)$ is a centered Gaussian with variance $\sigma^2$.

In \cite{EFL.23}, Enriquez, Faraud, and Lemaire analyze the behaviour of the path $(L_n^\ER(\lambda);\lambda>1)$ by seeing how the small components merge into the unique giant. Therein they show that the dynamics for the giant are approximately the same as those of an explicit stochastic differential equation which they can solve. In particular, they show
\begin{theorem}[Enriquez, Faraud and Lemaire \cite{EFL.23}]\label{thm:EFLthm}
    Let $B$ be a standard Brownian motion and set
    \begin{equation*}
        u(\lambda) = \frac{1}{1-\varrho^{\textup{ER}}(\lambda)}-\lambda \qquad\textup{and}\qquad v(\lambda) = \frac{1}{1-\varrho^{\textup{ER}}(\lambda)}-1 = \frac{\varrho^{\textup{ER}}(\lambda)}{1-\varrho^{\textup{ER}}(\lambda)}.
    \end{equation*}
    Then in the Skorohod space $\D((1,\infty),\R)$
    \begin{equation*}
    \left( \frac{L_n^\ER(\lambda) - \varrho^{\textup{ER}}(\lambda)n}{n^{1/2}};\lambda>1\right)  \weakarrow \left(\frac{1}{u(\lambda)} B(v(\lambda));\lambda >1\right).
    \end{equation*}
\end{theorem}

Recently, Corujo, Lemaire, and Limic \cite{CLL.24} established the same result using the ``simultaneous breadth-first walks'' of Limic \cite{Limic.19}. See also \cite{BBS.24}. In this article, we show that the giant of dynamic rank-one random graphs have similar Gaussian fluctuations whenever the weight distribution converges weakly and its second moment converges.

\subsection{Set-up and Main Result}

Consider a length $n$ weight vector $\bw = \bw^{(n)} = (w_1,\dotsm, w_n)$ where each $w_j = w_j^{(n)}>0$. Let $(\G_n(\bw,\lambda);\lambda\ge 0)$ be the random graph on $[n]:=\{1,2,\dotsm,n\}$ where each edge $\{i<j\}$ is included in $\G_n(\bw,\lambda)$ if and only if $E_{ij}\le \lambda/n$ where $(E_{ij}\sim\Exp(w_iw_j);i<j)$ are a collection of independent exponential random variables with respective rates $w_iw_j$.  Observe that
\begin{equation*}
    \PR(i\sim j\textup{ in }\G_n(\bw,\lambda)) = 1-\exp(-\lambda w_iw_j/n).
\end{equation*}
The Erd\H{o}s-R\'{e}nyi random graph corresponds to choice $\bw = (1,1,\dotsm,1)$. 
We call $w_i$ the \textit{weight} of vertex $i$ and for a connected component $\cC\subset \G_n(\bw,\lambda)$ we say its \textit{volume} is $\sum_{j\in \cC} w_j$. 

Letting $W_n$ be the weight of a uniformly chosen vertex in $\G_n(\bw,\lambda)$, we will assume throughout the article the following holds. 
\begin{assumption}\label{ass:1}
There is a random variable $W$ such that
\begin{align*}
    W_n\weakarrow W\qquad\textup{ and }\qquad \E[W_n^2] \to \E[W^2]<\infty. 
\end{align*}
\end{assumption}

Fix a random variable $W$ with finite variance and $\PR(W>0) = 1$. It is well-known, \cite{CL.02,BJR.07,AL.98} that when the degree distribution $W_n\to W$ in the Wasserstein-2 distance then the critical time $\lambda_\crit$ for the emergence of a giant connected component of the above graph occurs at $\lambda_\crit := 1/\E[W^2]$. That is the largest connected component of $\G_n(\bw,\lambda)$ has a positive proportion of all the vertices if and only if $\lambda>\lambda_\crit$. To describe the fluctuations of the giant, define for $p = 0,1$ and $n\ge 1$
\begin{equation}
\label{eqn:varphiDef}\varphi^{(n)}_p(t) = \E[W_n^p(1-e^{-W_nt})] = \sum_{j=1}^n n^{-1} w_j^p (1-e^{-w_jt}).
\end{equation} Let us also define $\theta^{(n)}$, $\varrho^{(n)}$, $\beta^{(n)}$ by
\begin{align}\label{eqn:thetadef}
    &\theta^{(n)}(\lambda) = \inf\{t>0: \varphi_1^{(n)}(\lambda t)-t<0\},\quad\quad \varrho^{(n)}(\lambda)  = \varphi_0^{(n)}(\lambda \theta^{(n)}(\lambda))\\
    \label{eqn:betadef}
    &\beta^{(n)}(\lambda) = 1-\lambda\E[W_n^2 e^{-W_n\lambda \theta^{(n)}(\lambda) }].
\end{align} Similarly, define $\varphi_p, \theta, \varrho,\beta$ for the analogous quantities where $W_n$ is replaced by $W$. Doing so, it is easy to see that $\lambda>\lambda_\crit = 1/\E[W^2]$ if and only if $\theta(\lambda) > 0$ if and only if $\varrho(\lambda)>0$. It is also easy to see that $\theta(\lambda)< \E[W]$, $\varrho(\lambda)< 1$. Observe that for $p = 0,1,$ that $\varphi_p^{(n)}$ are non-decreasing and converge pointwise to $\varphi_p$ under Assumption \ref{ass:1}. Therefore, $\varphi_p^{(n)}\to \varphi_p$ uniformly by Dini's theorem. Similarly, $\theta^{(n)}$ are a non-decreasing and, from the convergence of $\varphi_1^{(n)}$, one can see that $\theta^{(n)}\to \theta$ uniformly as well. From here it follows that $\varrho^{(n)}\to \varrho, \beta^{(n)}\to \beta$ uniformly as well.

Our main theorem is the following. We let $\cC_n(\lambda)$ be the connected component that has the largest volume. We let $L_n(\lambda)$ be the cardinality of $\cC_n(\lambda)$ and we let $V_n(\lambda)$ be its volume:
\begin{equation}\label{eqn:voldef}
    V_n(\lambda) = \sum_{j\in \cC_n(\lambda)} w_j.
\end{equation}
\begin{theorem}\label{thm:NRgiant}
   Under Assumption \ref{ass:1}, 
   \begin{align} \label{eqn:xexpand}
       \left(\left(\frac{L_n(\lambda) - \varrho^{(n)}(\lambda)n}{n^{1/2}}, \frac{V_n(\lambda)-\theta^{(n)}(\lambda)n}{n^{1/2}}\right) ;\lambda>\lambda_\crit\right)\weakarrow \left(\bX(\lambda);\lambda>\lambda_\crit\right)
   \end{align}
   where $\bX$ is the $\R^2$-valued centered, continuous Gaussian process 
   \begin{align*}
   \bX(\lambda) = \left(  
   \Psi_0(\lambda \theta(\lambda)) + \frac{\lambda \varphi_0'(\lambda\theta(\lambda))}{\beta(\lambda)}\Psi_1(\lambda\theta(\lambda)), \frac{1}{\beta(\lambda)}\Psi_1(\lambda \theta(\lambda))\right)
   \end{align*}
   for centered continuous Gaussian processes $\Psi_0,\Psi_1$ with covariance
   \begin{align*}
       \E[\Psi_p(s)\Psi_q(t)] = \E[W^{p+q} e^{-Ws}(1-e^{-Wt})]\qquad\forall s\le t, \textup{ and }p,q\in\{0,1\}.
   \end{align*}
\end{theorem}

\subsection{On CLTs for Giants}

Before continuing to the proof, let us mention some of the various central limit theorems appearing in the literature for the size or volume of the giant in random graph models.

For the Erd\H{o}s-R\'enyi random graph, the first CLT was obtained by Stepanov in \cite{Stepanov.70}; however, Stepanov's formulation is not quite as it appears in \eqref{eqn:CLT}. The form as it appear in \eqref{eqn:CLT} is due to Pittel \cite{Pittel.90}. A similar formulation, without a direct random graph connection, was obtained by Martin-L\"of in \cite{MartinLof.86}. In \cite{BBF.00}, Barraez, Boucheron, and Fernandez de la Vega used a random walk and the martingale CLT to provide a rate of convergence for the CLT in \eqref{eqn:CLT}. In \cite{Rath.18}, R\'{a}th obtained an explicit generating function for the size of the connected component containing the vertex $1$ and used this to obtain the CLT in \eqref{eqn:CLT}.

In \cite{PW.05}, Pittel and Wormald investigated the trivariate random vector of the number of vertices contained in at least one cycle in the giant (the 2-core), the number vertices not in a cycle of a giant (the tree mantle), and the number surplus edges. Therein they show a trivariate CLT with an explicit covariance matrix. The results of \cite{PW.05} also extend to the barely supercritical regime: $\lambda = \lambda_\crit + \eps_n$ where $n^{1/3}\eps_n\to \infty$. A random walk approach for the size of the giant was obtained by Bollob\'{a}s and Riordan in \cite{BR.12}. In \cite{Puhalskii.05}, Puhalskii obtained a trivariate CLT for the number of connected components, the number of vertices in the giant, and the surplus of the giant for the supercritical random graph. Puhalskii also obtained a CLT for the number of connected components below criticality as well. 

For the dynamic Erd\H{o}s-R\'{e}nyi random graph, as already mentioned there is Theorem \ref{thm:EFLthm} obtained by Enriquez, Faraud and Lemaire \cite{EFL.23}. A new proof, analogous to the method taken in this work, was obtained by Corujo, Lemaire and Limc \cite{CLL.24}. Their approach allows them to also obtain a functional CLT (FCLT) for the number of vertices in the giant in the barely supercritical regime $\lambda = \lambda_\crit + t\eps_n$ for $t>0$ and $n^{1/3}\eps_n\to\infty$. A trivariate FCLT for the number of connected components, the size of the giant, and its surplus was obtained in \cite{BBS.24} by Bhamidi, Budhiraja and Sakanaveeti. The latter extends to finite type stochastic blockmodels and FCLTs for the number of connected components below criticality. This extends the work of Corujo \cite{Corujo.24}, which also allows for inhomogeneous rank 1 models similar to our setting.

For static inhomogeneous models there are also numerous central limit theorems. In \cite{MartinLof.86}, Martin-L\"of obtained a CLT for the number of individuals infected in a generalized stochastic epidemic. Under certain conditions, see e.g. \cite{Neal.03}, this result can be reformulated into a CLT for the giant of a random graph model. In \cite{Seierstad.13}, Seierstad provided fairly general conditions for the number of vertices not contained in a tree for a dynamic random graph processes to satisfy a central limit theorem. Often, such vertices are almost exclusively in the giant. 

Closely related to our results and approach here is the work by Neal in \cite{Neal.07}. Neal obtained a CLT for the number of individuals infected in a ``variable generalized stochastic epidemic.'' In this model, each infected individual $i\in[n]$ will infect a susceptible individual $j\in[n]$ independently with probability $1-\exp(-Q_iD_j/n)$ where $(Q_i,D_i)_{i\in[n]}$ are i.i.d. random variables. One can think of $D_j/n$ as the rate at which an individual $j$ makes contact with every other individual in the population and $Q_i$ is the duration of time that person $i$ is infectious. Essentially, Neal relates the epidemic to the process
\begin{align*}
    X_n^{(Q,D)}(t) = -t+\sum_{j=1}^n Q_j \bone_{[\xi_j\le D_j t]}
\end{align*}
where $\xi_j$ are i.i.d. rate 1 exponential random variables independent of $Q_j,D_j$. This connection with i.i.d. exponential random variables was originally described by Sellke in \cite{Sellke.83}, and is closely related to the encoding of the multiplicative random graph described below by Limic \cite{Limic.19}. However, Limic's construction is dynamic while Sellke's is static. A similar approach to Neal was taken in \cite{Startsev.01}.

While not directly related to our model under consideration here, Barbour and R\"ollin \cite{Barbour.19} established a CLT for the largest connected component of the super-critical configuration model. Janson \cite{Janson.20} extended this to the configuration model conditioned on simplicity (that is no self-loops nor multiple edges). The asymptotic variance of the giant was obtained by Ball and Neal \cite[Theorem 2.1]{BN.17}. See also \cite{Riordan.12} for the barely supercritical regime.

As mentioned above, \cite{CLL.24} is able to analyze the barely supercritical behavior of the giant for the Erd\H{o}s-R\'{e}nyi. We believe that different assumptions must be placed on the weight vector $\bw^{(n)}$ to analyze the barely supercritical behavior of $\G_n(\bw,\lambda)$. This is because of the qualitatively different behavior of the near-critical random graphs whenever $W_n$ has finite third moment or infinite third moment \cite{Aldous.97,AL.98}. 
\subsection*{Acknowledgements} The author would like to thank Josu\'e Corujo and Vlada Limic for sharing a note that eventually evolved into the paper \cite{CLL.24} with S. Lemaire. The approach of Donsker's theorem drastically simplified the author's original approach using integral equations. The author would also like to thank Terry Harris for helpful discussions on the inverse and implicit function theorems.

\section{Simultaneous Breadth-first walks}

\subsection{Breadth-first walks and random graphs}

Consider the weighted graph $G^{\MC}(\bx,q)$ consisting of vertices $1,\dotsm,n$ where each vertex $i$ has weight $x_i>0$ and each edge is included independently with probability $1-e^{-x_ix_jq}$. We make a small change in parametrization to better match \cite{Limic.19}. We let $V_n^\circ(\bx,q)$ denote the maximal volume of any connected component connected component:
\begin{equation*}
    V_n^\circ(\bx,q) = \max_{C\subset \G_n(\bx,q)} \sum_{i\in C} x_i
\end{equation*} where the maximum is taken over all connected components of $G^{\MC}(\bx,q)$. In terms of the multiplicative coalescent of Aldous \cite{Aldous.97}, this is simply the mass of the largest block at time $q$. 

We make use of the following theorem of Limic \cite{Limic.19}, which requires a little notation. We say that $(g,d)$ is an excursion interval of a function $f(t)$ (with no negative jumps and null at zero) if
\begin{equation*}
    f(g-) = \inf_{s\le g}f(s)=f(d) \qquad \textup{and}\qquad f(t-) > f(g-)\textup{ for all }t\in(g,d).
\end{equation*} In particular, there is some $y$ such that $g = \inf\{t:\inf_{s\le t}f(s) = -y\}$ while $d = \inf\{t:f(t)<-y\}$.
Given a function $f$, let $\LL(f)$ denote the length of the longest excursion interval of the function $f$. If there are multiple longest excursions, $\LL$ with the length of the first one. 

\begin{theorem}[Limic {\cite{Limic.19}}]\label{thm:limic}
Fix $\bx = (x_1,\dotms,x_n)$. Let $\xi^\circ_1,\dotsm,\xi^\circ_n$ be independent exponential random variables with $\xi_j^\circ\sim \Exp(x_j)$ (that is rate $x_j$ and mean $1/x_j$). Define
    \begin{equation*}
        Z_n(t,q;\bx) = \sum_{j=1}^n x_j \bone_{[\xi^\circ_j\le q t]} - t.
    \end{equation*}
    Then $\left(\LL(t\mapsto Z_n(t,q;\bx));q > 0\right) \overset{d}{=} \left(V^\circ_n(\bx,q);q >0\right).
   $
\end{theorem}

We now consider the dynamic Poisson random graph $\G_n(\bw,\lambda)$ where each edge $e = \{i,j\}$ is included independently with probability $1-\exp(-\lambda w_iw_j /n)$. Let us denote by $W_n$ the weight of a uniformly chosen vertex in $\G_n(\bw,\lambda)$. Let $(\xi_j;j=1,\dotsm,n)$ be independent $\Exp(w_j)$ random variables and define the processes $X_{n,p} = (X_{n,p}(t);t\ge 0)$ by 
\begin{equation}\label{eqn:xrep}
    X_{n,p}(t) = \sum_{j=1}^n n^{-1}w_j^p\bone_{[\xi_j\le t]}.
\end{equation} The choice of scaling is for a matter of convenience for later scaling limits and note that by setting $\xi_j^\circ = n^{2/3}\xi_j$ and $x_j = n^{-2/3}w_j$ 
\begin{equation*}
    n^{-1/3} Z_n(n^{1/3} t,\lambda; n^{-2/3} \bw) = X_{n,1}(\lambda t) - t.
\end{equation*}
We get the following corollary of Theorem \ref{thm:limic}.
\begin{corollary}\label{cor:giant}
    Let $(\g_n(\lambda),\dd_n(\lambda))$ denote the longest excursion interval of $t\mapsto X_{n,1}(\lambda t)-t$. Then
    \begin{align*}
       & \left(\left(n^{-1}V_n(\lambda), n^{-1}L_n(\lambda)\right);\lambda>0\right) \overset{d}{=} \left(\dd_n(\lambda)-\g_n(\lambda), X_{n,0}(\lambda \dd_n(\lambda)) - X_{n,0}(\lambda \g_n(\lambda)) + n^{-1};\lambda >0\right)
    \end{align*} where $V_n$ (resp. $L_n$) are the volume (resp. cardinality) of the most voluminous component of $\G_n(\bw,\lambda)$.
\end{corollary} 
\begin{proof}
    The statement of the volume is direct consequence of the remarks preceding the statement of the corollary. For the cardinality note that weights $w_j$ contributing to the largest excursion are precisely those $w_j$ such that 
    \begin{equation*}
        \lambda \g_n(\lambda)\le  \xi_j \le \lambda \dd_n(\lambda).
    \end{equation*} The number of such $j$ is given by 
    \begin{align*}
        \sum_{j: \lambda \g_n(\lambda)\le \xi_j\le \lambda\dd_n(\lambda)}& 1 = n\left( X_{n,0}(\lambda \dd_n(\lambda)) - X_{n,0}(\lambda \g_n(\lambda)-)\right) = n\left( X_{n,0}(\lambda \dd_n(\lambda)) - X_{n,0}(\lambda \g_n(\lambda))+1\right)
    \end{align*}
    as desired.
\end{proof}

\subsection{Weighted Empirical Processes}

Let us recall a result of Shorack \cite{Shorack.79} concerning weighted empirical processes for triangular arrays. To start with, let $(\beta_{j,n} ;j\in [n],n\ge 1)$ be a triangular array of row-independent random variables with values in $\R_+$ and with respective cumulative distribution functions $G_{j,n}(t)$. Let $c_{j,n}$ be (deterministic) weights and define
\begin{align*}
    F_n(t) = n^{-1} \sum_{j=1}^n c_{j,n}\bone_{[\beta_{j,n}\le t]} \qquad \textup{and}\qquad \overline{F}_n(t) = n^{-1} \sum_{j=1}^n c_{j,n}G_{j,n}(t)
\end{align*}
as the weighted empirical process and its expectation, respectively. Lastly, for a non-decreasing function $f$ we write $f(\infty) = \lim_{t\to\infty} f(t)$. We can now state a theorem of Shorack \cite{Shorack.79}.
\begin{theorem}[Convergence of weighted empirical processes]\label{thm:Shorack}
    Consider the above set-up. Further suppose
    \begin{enumerate}
        \item[\textbf{(i)}]  For each $t\in [0,\infty]$,
        $ n^{-1}\sum_{j=1}^n c_{j,n}^2 G_{j,n}(t) \to \nu_{\infty}(t)
        $
        for some non-decreasing function $\nu_\infty(t)$.
        \item[\textbf{(ii)}] As $n\to\infty$ it holds that $\max_{j\in[n]} c_{j,n}^2 = o(n)$.
        \item[\textbf{(iii)}] There is some limit $K:\R_+^2\to\R$ such that the weighted covariance
        \begin{align*}
            K_n(s,t) := n^{-1}\sum_{j=1}^n c_{j,n}^2 \left(G_{j,n}(s\wedge t) - G_{j,n}(s)G_{j,n}(t)\right) \to K(s,t)
        \end{align*}
        for all $s,t\in \R_+$.
    \end{enumerate}
    Then in the Skorohod $J_1$ topology
    \begin{align*}
       \left( n^{1/2}\left(F_n(t)-\overline{F}_n(t)\right);t\ge 0\right)\weakarrow (\Psi(t);t\ge 0)
    \end{align*}
    for a centered Gaussian process with continuous sample paths and covariance function $K(s,t)$.
\end{theorem}
\begin{remark}
    We note that in \cite{Shorack.79}, Shorack assumes that $\sum_{j=1}^n c_{j,n}^2 = n$ and $\beta_{j,n}$ take values in $[0,1]$. The first can be dropped by scaling and noting that $\sum_{j=1}^n c_{j,n}^2 = \sum_{j=1}^n c_{j,n}^2 G_{j,n}(\infty)$. The second can be dropped by considering $\alpha_{j,n}$ defined by 
        $\beta_{j,n} = -\log(1-\alpha_{j,n})$
    where $\alpha_{j,n}$ are supported on $[0,1]$.
\end{remark}

\subsection{Convergence of $X_{n,p}$}

In this section, we prove the following theorem.
\begin{theorem}\label{thm:Xconv}Suppose Assumption \ref{ass:1}. Then, jointly for $p = 0,1$,
    \begin{equation}\label{eqn:xnhelp1}
       \left(n^{1/2}\left( X_{n,p}(t) - \varphi_p^{(n)}(t)\right);t\ge 0\right)\weakarrow \left(\Psi_p(t);t\ge 0\right)
    \end{equation}
    where $\Psi_p$ are centered Gaussian processes with covariance function
    \begin{align}\label{eqn:covar}
       \E[\Psi_p(s)\Psi_q(t)] = \E\left[W^{p+q} \left(e^{-W(t\vee s)} - e^{-W(s+t)}\right)\right].
    \end{align}
\end{theorem}
\begin{proof}
    Recall the definition of $X_{n,p}$ from \eqref{eqn:xrep}. Clearly, $X_{n,p}$ is a weighted empirical process with $c_{j,n} = w_j^p$ where we remind the reader that $w_j = w_j^{(n)}$ also depends on $n$. Also, $\E[X_{n,p}(t)] = \varphi_p^{(n)}(t)$.

    We begin by with the marginal convergence for each fixed $p = 0,1$ by verifying the conditions of Theorem \ref{thm:Shorack}. For hypothesis \textbf{(i)} in Theorem \ref{thm:Shorack} we note that, in our case, $G_{j,n}(t) = 1-\exp(-w_jt)$ for all $t\ge 0$ and so
     \begin{align*}
         n^{-1} \sum_{j=1}^n c_{j,n}^2 G_{j,n}(t) = n^{-1}\sum_{j=1}^n w_j^{2p} (1-\exp(-w_jt)) = \E[W_n^{2p}(1-e^{-W_nt})].
     \end{align*}
     As $W_n\weakarrow W$ and $\E[W^2_n]\to \E[W^2]$, we have for all $t\in[0,\infty]$
     \begin{align*}
         \E[W_n^{2p}(1-e^{-W_nt})] \to \E[W^{2p}(1-e^{-W t})]\qquad\textup{ for }p = 0,1.
     \end{align*} 

For hypothesis \textbf{(ii)} in Theorem \ref{thm:Shorack} we reorder $\bw^{(n)}$ so that $w_1>w_2>\dotsm$. Clearly this hypothesis holds when $p = 0$, so we restrict to $p = 1$. Note that for all $t\ge 0$
    \begin{equation*}
    \limsup_{n\to\infty} n^{-1}w_1^2 \le   \limsup_{n\to\infty} n^{-1}\sum_{j: w_j>t} w_j^2 = \limsup_{n\to\infty}   \E[W_n^2\bone_{[W_n> t]}] \le  \E[W^2\bone_{[W\ge t]}]\underset{t\to\infty}{\longrightarrow}0.
    \end{equation*} It follows that $c_{1,n}^2 = w_1^2 = o(n)$.

We now turn to hypothesis \textbf{(iii)} in Theorem \ref{thm:Shorack}. We write $K_{n,p}$ for the respective weighted covariance term appearing when verifying the conditions for $X_{n,p}$. We recall that $G_{j,n}(u) = 1-e^{-w_ju}$ for all $u\ge 0$ and $p = 0,1$ and so for $s\le t$ 
\begin{align}\label{eqn:covconv}
   K_{n,p}(s,t) &= n^{-1}\sum_{j=1}^n w_{j}^{2p} G_{j,n}(s)(1-G_{j,n}(t))=n^{-1} \sum_{j=1}^{n} w_j^{2p} (1-e^{-w_js})e^{-w_jt}\\
  \nonumber &= \E[W_n^{2p} \left(e^{-W_nt} - e^{-W_n(t+s)}\right)] \to \E[W^{2p}(e^{-Wt} - e^{-W(t+s)})]
\end{align}
as claimed. 

Hence, we have for each $p = 0,1$ the marginal weak convergence in \eqref{eqn:xnhelp1} holds for the stated limit. We now verify that this convergence is joint over $p = 0,1$. Clearly, the marginal convergence implies that the sequence \begin{align*}
  \left\{ \left( \left(n^{1/2}\left( X_{n,0}(t) - \varphi_0^{(n)}(t)\right);t\ge 0\right)  ,   \left(n^{1/2}\left( X_{n,1}(t) - \varphi_1^{(n)}(t)\right);t\ge 0\right)\right);n\ge 1\right\}
\end{align*}
is tight in $\D(\R_+,\R)^2$. We just need to show that all the subsequential limits are the same; however since the limits are centered continuous Gaussian processes, we only need to show convergence of the covariance. Namely, we must show
\begin{align*}
  n\E\left[\left(X_{n,0}(s)-\varphi_{0}^{(n)}(s)\right)\left(X_{n,1}(t)-\varphi_{1}^{(n)}(t)\right)\right]\to \E[W(e^{-W(t\vee s)}- e^{-W(t+s)})].
\end{align*} This is essentially the same computation as \eqref{eqn:covconv} so we omit this.
\end{proof}

\section{Proof of Theorem \ref{thm:EFLthm}}

We now turn to the proof of Theorem \ref{thm:EFLthm}. By Theorem \ref{thm:Xconv} and an application of Skorohod's representation theorem, we have the following approximation for $H_n(t,\lambda):= X_{n,1}(\lambda t)-t$
\begin{align}
\label{eqn:Hnexpand} H_n(t,\lambda)  =  (\varphi_1^{(n)}(\lambda t)-t) +\frac{1}{n^{1/2}} \Psi_1(\lambda t) + \Delta_n(\lambda t)
\end{align}
where $n^{1/2}\Delta_n(\lambda t) \to 0$ locally uniformly. 
Here we have used the fact that $J_1$ convergence is equivalent to local uniform convergence when the limit is continuous. See, for example, \cite[Section 3.10]{EK.86}.

Recall from around \eqref{eqn:thetadef} that $\theta^{(n)}(\lambda)$, for $\lambda>\lambda_{\crit} = 1/\E[W^2]$ and all $n$ large, is the unique root in $t>0$ of $\E[W_n(1-e^{-W_n\lambda t})]-t$. It is easy to see that the longest excursion of $t\mapsto H_n(t,\lambda)$ is roughly $\theta^{(n)}(\lambda)$ for each fixed $\lambda>\lambda_\crit$ and $n$ large. More precisely, let $\g_n(\lambda)$ (resp. $\dd_n(\lambda)$) be the start (resp. end) of the longest excursion of $t\mapsto H_n(t,\lambda)-t$, then for each fixed $\lambda$ 
\begin{equation*}
    \g_n(\lambda) \weakarrow 0 \qquad\textup{and}\qquad \dd_n(\lambda)\weakarrow \theta(\lambda).
\end{equation*}
We will show the following:
\begin{theorem}\label{thm:DGconv}
  Under Assumption \ref{ass:1}, the following weak convergence holds in $\D((\E[W^2],\infty),\R)$
    \begin{align*}
        \left(n^{1/2}\left(\dd_n(\lambda) - \theta^{(n)}(\lambda) \right);\lambda>\lambda_\crit\right) &\weakarrow \left(\frac{\Psi_1(\lambda \theta(\lambda))}{\beta(\lambda)};\lambda >\lambda_\crit\right)\\
    \left(n^{1/2}\g_n(\lambda);\lambda>\lambda_\crit\right)&\weakarrow (0;\lambda>\lambda_\crit),
    \end{align*}
    where $\Psi_1$ is as in Theorem \ref{thm:Xconv} and $\beta$ is as in \eqref{eqn:betadef}. Moreover, this convergence holds jointly with the convergence in Theorem \ref{thm:Xconv}.
\end{theorem}

The proof of the above theorem is the entirety of Section \ref{sec:proofofthm3.1}. Before turning to that, we show how Theorem \ref{thm:DGconv} implies Theorem \ref{thm:NRgiant}.
\subsection{Proof of Theorem \ref{thm:NRgiant} by Theorem \ref{thm:DGconv}}

The weak convergence in Theorem \ref{thm:NRgiant} is simplified by the following elementary lemma, which is a slight variation of Theorem 13.3.3 in \cite{Whitt.02}. 
\begin{lemma}\label{lem:whitt}
  Let $x_n,y_n\in\D(\R_+,\R)$, and $x,y,u,v\in C(\R_+,\R)$ be such that $y$ is strictly increasing and $x$ is continuously differentiable. If $c_n\to\infty$ is a sequence of real numbers such that locally uniformly in $t$ it holds that $c_n(x_n(t)-x(t)) \to u(t)$ and $c_n(y_n(t)-y(t)) \to v(t)$, then
  \begin{align*}
      c_n(x_n\circ y_n(t) - x\circ y(t)) \longrightarrow u\circ y(t) + x'\circ y(t) v(t)\qquad\text{locally uniformly}.
  \end{align*}
\end{lemma}
The only difference between the result stated above and the statement in \cite[Theorem 13.3.3]{Whitt.02} is that Whitt assumes that $y_n$ is non-decreasing so that the composition $x_n\circ y_n$ is c\`adl\`ag. This allows the use of the $J_1,M_1$, or $M_2$ topologies. However, the same proof found therein works when all convergences are assumed to be locally uniform.

Recall from Corollary \ref{cor:giant}, that
\begin{equation*}
    \left(n^{-1} L_n(\lambda) ;\lambda>0\right)\overset{d}{=} \left(n^{-1}+X_{n,0}\left( \lambda \dd_n(\lambda)\right)- X_{n,0}\left(\lambda \g_n(\lambda)\right);\lambda>0\right).
\end{equation*}
The $n^{-1}$ term is asymptotically negligible after scaling by $n^{1/2}$ and will subsequently be ignored. By continuity of multiplication, it is easy to see from Theorem \ref{thm:DGconv} that (jointly with the weak convergence of $X_{n,0}$)
\begin{equation*}
    \left(n^{1/2}\left(\lambda \dd_n(\lambda) - \lambda \theta^{(n)}(\lambda)\right);\lambda>\lambda_\crit\right)\weakarrow \left(\frac{\lambda}{\beta(\lambda)} \Psi_1(\lambda \theta(\lambda))); \lambda>\lambda_\crit\right).
\end{equation*}

Recall that $\varrho^{(n)}(\lambda) = \varphi_0^{(n)}(\lambda \theta^{(n)}(\lambda))$. By combining Theorems \ref{thm:Xconv} and \ref{thm:DGconv} it is easy to see that jointly
\begin{align*}
    & \left(n^{1/2} \left(X_{n,0}(\lambda \dd_n(\lambda)) - \varrho^{(n)}(\lambda)\right);\lambda>\lambda_\crit\right) \weakarrow \left(\Psi_0(\lambda\theta(\lambda))+ \varphi_0'(\lambda \theta(\lambda)) \frac{\lambda \Psi_1(\lambda\theta(\lambda))}{\beta(\lambda)};\lambda>\lambda_\crit\right)\\
&\left(n^{1/2} X_{n,0}(\lambda \g_n(\lambda));\lambda>\lambda_\crit\right) \weakarrow \left(0;\lambda>\lambda_\crit\right).
\end{align*}
Combining the last few displayed equations and Lemma \ref{lem:whitt}, we see (jointly with the convergence in Theorem \ref{thm:DGconv})
\begin{align*}
    \left(\frac{L_n(\lambda)-\varrho^{(n)}(\lambda)n}{n^{1/2}};\lambda>\lambda_\crit\right) \weakarrow \left(\bX_0(\lambda);\lambda>\lambda_\crit\right),
\end{align*}
where $\bX_0 (\lambda) = \Psi_0 (\lambda \theta(\lambda)) + \frac{\lambda \varphi_0'(\lambda \theta(\lambda))}{\beta(\lambda)} \Psi_1(\lambda\theta(\lambda)).$ The convergence in Theorem \ref{thm:NRgiant} as a process in $\D((\lambda_\crit,\infty),\R^2)$ follows from both the limits $\bX_0$ and the limit from Theorem \ref{thm:DGconv} being a.s. continuous and elementary properties of local uniform convergence.

\subsection{Proof of Theorem \ref{thm:DGconv}}\label{sec:proofofthm3.1}

Our approach will be analogous to \cite{BR.12}. We recall from \eqref{eqn:betadef} that 
\begin{equation*}
    \beta^{(n)}(\lambda)= -\frac{\partial}{\partial t}\left(\varphi_1^{(n)}(\lambda t)-t\right)\bigg|_{t = \theta^{(n)}(\lambda)}=
     \left(1-\lambda \E[W_n^2 e^{-W_n \lambda \theta^{(n)}(\lambda)}]\right)
\end{equation*} and a similar statement holds for $\beta$. We claim that for all $\lambda>\lambda_\crit$, $\beta(\lambda)>0$. To see this note that for each fixed $\lambda>\lambda_\crit$ the function
$f:t\mapsto \varphi_1(\lambda t)-t$ is bounded above and continuously differentiable with strictly decreasing derivative $f'(t) = \lambda \E[W^2 e^{-W \lambda t}] - 1$. Since $f$ is maximized uniquely at a point $t^*$ strictly between its two roots $0<\theta(\lambda)$, the derivative at $\theta(\lambda)>t^*$ must be strictly negative as $f'(\theta(\lambda)) < f'(t^*) = 0$. Finally, $\beta(\lambda) = -f'(\theta(\lambda)) >0$.

\begin{proposition}\label{prop:GD} Fix $\lambda_\crit<\lambda_0\le \lambda_1$, $\eps>0$, and set $T = \E[W]$. Under the a.s. coupling of $H_n$ in \eqref{eqn:Hnexpand}, it holds that a.s. for all sufficiently large $n$
    \begin{align*}
        \inf_{\lambda\in[\lambda_0,\lambda_1]} \inf_{\displaystyle s\in\left[\frac{\eps}{n^{1/2} }, \theta^{(n)}(\lambda) + \frac{ \Psi_1(\lambda\theta^{(n)}(\lambda)) - \eps}{\beta^{(n)}(\lambda) n^{1/2}}\right]} H_n(s,\lambda) > 0.
    \end{align*}
\end{proposition}

\begin{proof}

Note that $\frac{d}{dt}\varphi_1^{(n)}(t) = \E[W_n^2 e^{-W_nt}]$ and $\E[W_n^2e^{-W_nt}]\to \E[W^2 e^{-Wt}]$ uniformly in $t\ge 0$ by Dini's theorem. Also, by the mean value theorem we have for all $x\ge 0$, $y\ge - \theta^{(n)}(\lambda)$ and $\lambda >\lambda_\crit$
\begin{align}
    \nonumber &\varphi_1^{(n)}(x) = \E[W_n^2 e^{-W_n \zeta_{n,l}(x)}]x\\
    \label{eqn:taylor2}&\varphi_1^{(n)}(\lambda \theta^{(n)}(\lambda) + y) - \varphi_1^{(n)}(\lambda \theta^{(n)}(\lambda))= \E[W_n^{2} e^{-W_n \lambda \zeta_{n,r}(\lambda, y)}]y
\end{align} for some $\zeta_{n,l}(x)\in [0,x]$ and $\zeta_{n,r}(\lambda, y)$ between $\theta^{(n)}(\lambda)$ and $\theta^{(n)}(\lambda) + y$. Similar formulations hold for $\varphi_1$.  Recall that $\varphi_1^{(n)}(\lambda \theta^{(n)}(\lambda)) = \theta^{(n)}(\lambda)$ for $\lambda>\lambda_\crit$ and $n$ large.

Let us consider some sequence $a_n\to\infty$ such that $a_n = o(n^{1/2})$ and look at $s\in [\eps n^{-1/2}, a_nn^{-1/2}]$. We will also fix $\lambda\in[\lambda_0,\lambda_1]$. We have
\begin{align*}
    H_n&(s,\lambda) = \varphi_1^{(n)}(\lambda s) - s + n^{-1/2} \Psi_1(\lambda s) + \Delta_n(\lambda s)\\
    &= \left(\lambda \E[W_n^2 e^{-W_n \zeta_{n,l}(\lambda s)}] - 1\right)s + n^{-1/2} \Psi_1(\lambda s) + \Delta_n(\lambda s)
\end{align*} Let us look at the first term on the right-hand side. Since $\zeta_{n,l}(\lambda s) \le \lambda s \le \lambda_1 a_n n^{-1/2} = o(1)$, we see that as $n\to\infty$ and uniformly for $\lambda\in[\lambda_0,\lambda_1]$ and $s\in [\eps n^{-1/2}, a_nn^{-1/2}]$ that
\begin{align*}
    \lambda \E[W_n^2 e^{-W_n \zeta_{n,l}(\lambda s)}] - 1  \longrightarrow \lambda \E[W^2] - 1.
\end{align*}
For the second term, we use the continuity of $\Psi_1$ and the fact that $\Psi_1(0) = 0$ to see
\begin{align*}
    \inf_{\lambda\le \lambda_1, s\le a_n n^{-1/2}} \Psi_1(\lambda s) = o(1).
\end{align*} 
Hence, uniformly on $\lambda\in[\lambda_0, \lambda_1]$ and $s\in [\eps n^{-1/2}, a_nn^{-1/2}]$ we have with probability 1 that 
\begin{align}\label{eqn:taylr3}
    H_n(s,\lambda) \ge \left(\lambda \E[W^2]-1 +o(1)\right)\eps n^{-1/2} + o(1) n^{-1/2} >0\textup{ for all }n\textup{ large}.
\end{align}

We now turn to the right-edge. Let us now focus on $s\in [\eps,a_n]$. We have
\begin{align*}
    H_n&\Bigg(\theta^{(n)}(\lambda) + \overbrace{\frac{\Psi_1(\lambda \theta^{(n)}(\lambda)) -s}{n^{1/2}\beta^{(n) }(\lambda)}}^{x:=}, \lambda\Bigg)\\
    &= \varphi_1^{(n)}\left(\lambda \theta^{(n)}(\lambda) + \lambda x\right) - \left(\theta^{(n)}(\lambda) +x\right)+ n^{-1/2} \Psi_1 (\lambda (\theta^{(n)}(\lambda)+x)) + \Delta_n(\lambda \theta^{(n)}(\lambda) + \lambda x).
\end{align*} 
Using \eqref{eqn:taylor2} and an argument similar to the derivation of \eqref{eqn:taylr3} we have uniformly in $s\in [\eps,a_n]$ and $\lambda\in[\lambda_0,\lambda_1]$ that
\begin{align*}
    \varphi_1^{(n)}&(\lambda \theta^{(n)} (\lambda) + \lambda x) - \left(\theta^{(n)}(\lambda)+x\right)  = \left( \lambda \E[W^2_n e^{-W_n \lambda \zeta_{n,r}(\lambda,\lambda x)}] -1\right)x\\
    &=\left(\lambda \E[W_n^2 e^{-W_n\lambda \theta^{(n)}(\lambda)}]-1+o(1)\right)x= (-\beta^{(n)}(\lambda)+o(1))x\\
    &=  n^{-1/2} (s -\Psi_1(\lambda \theta^{(n)}(\lambda)) )(1+ o(1)).
\end{align*} For the last line we used $\frac{1}{\beta^{(n)}(\lambda)} = O(1)$ as $n\to\infty$.
Hence, uniformly for $\lambda\in [\lambda_0, \lambda_1]$ and $s\in [\eps,a_n]$
\begin{align*}
   n^{1/2} H_n&\left(\theta^{(n)} + \frac{\Psi_1(\lambda \theta^{(n)}(\lambda)) -s}{\beta^{(n)}(\lambda)  n^{1/2}},\lambda\right)\\
    &= s(1+o(1)) +  \left(\Psi_1(\lambda \theta^{(n)}(\lambda) + \lambda x) - \Psi_1(\lambda \theta^{(n)}(\lambda))+o(1)\right) + \Delta_n(\lambda \theta^{(n)}(\lambda) + \lambda x)\\
    &= s(1+o(1)) + o(1) > \frac\eps2
\end{align*} for all $n$ sufficiently large. We have used $x\to 0$ as $n\to\infty$.

What remains to show is 
\begin{align*}
    \inf_{\lambda\in[\lambda_0,\lambda_1]} \inf_{s\in I_n(\lambda)} H_n(s,\lambda)>0.
\end{align*}
where $I_n(\lambda) = \{\frac{a_n}{n^{1/2}} \le s \le \theta^{(n)}(\lambda) + \frac{\Psi_1(\lambda \theta^{(n)}(\lambda)) - a_n}{\beta^{(n)}(\lambda)n^{1/2}}\}$. However, this is relatively easy. Note that 
\begin{align*}
    \inf_{\lambda\in [\lambda_0,\lambda_1]}\inf_{s\in I_n(\lambda)}\Psi_1(\lambda s)  = O(1).
\end{align*}
So, uniformly in $\lambda\in[\lambda_0,\lambda_1], s\in I_n(\lambda)$, we have
\begin{align*}
    H_n(s,\lambda)  = \varphi^{(n)}_1(\lambda s)- s + O(n^{-1/2})
\end{align*}
where the constants in the $O(n^{-1/2})$ are random but uniform over the specified collection of $\lambda,s$. The estimates in \eqref{eqn:taylor2} can easily be used to show that the above quantity is strictly positive. 
\end{proof}

The above proposition implies that for all for all $\lambda_\crit<\lambda_0\le \lambda_1$ and $\eps>0$ that a.s. under the coupling in \eqref{eqn:Hnexpand} 
\begin{equation}\label{eqn:supset}
(\g_n(\lambda),\dd_n(\lambda)) \supset \left[\eps n^{-1/2}, \theta(\lambda) + \frac{1}{\beta^{(n)}(\lambda) n^{1/2}} \left({\Psi_1(\lambda\theta^{(n)}(\lambda))} -\eps \right)\right]
\end{equation} for all $\lambda\in[\lambda_0,\lambda_1]$ and $n$ sufficiently large. This immediately yields the following.
\begin{corollary}\label{cor:Gconvto0}
For the coupling in \eqref{eqn:Hnexpand}, $n^{1/2}\g_n(\lambda)\to 0$ locally uniformly in $\lambda>\lambda_\crit$.
\end{corollary}

A reverse inclusion, analogous to \eqref{eqn:supset}, requires more work. Indeed, a.s. for all $\lambda>\lambda_\crit$ function $t\mapsto H_n(t,\lambda)<0$ for all $t$ sufficiently small and so the largest excursion of $H_n(t,\lambda)$ is not an excursion above zero but some negative level (which depends on $\lambda$). However, since $H_n(t,\lambda) = X_{n,1}(\lambda t) - t \ge -t$ for all $t$ this excursion is above level $o(n^{-1/2})$. More precisely, from \eqref{eqn:supset}:

\begin{lemma}\label{lem:gnlowbound}
    Fix $\lambda_\crit<\lambda_0\le \lambda_1$ and set $T = \E[W]$. Then a.s. under the coupling in \eqref{eqn:Hnexpand}, for any $\eps>0$ we have
    \begin{equation}\nonumber
\inf_{\lambda\in[\lambda_0,\lambda_1]} H_n(\g_n(\lambda),\lambda) \ge  - \frac{\eps}{ n^{1/2}}
    \end{equation}
    for all $n$ sufficiently large.
\end{lemma}

We now turn to some quantitative control over $H_n(t,\lambda)$ near the right edge. 

\begin{lemma}\label{lem:Gnrightedge}
    For any $\lambda_\crit<\lambda_0\le \lambda_1$ and set $T = \E[W]$. Then a.s. under the coupling in \eqref{eqn:Hnexpand}, for any $\eps>0$ we have
    \begin{align*}
       \sup_{\lambda\in[\lambda_0,\lambda_1]}  H_n\left(\theta(\lambda) + \frac{1}{n^{1/2}} \frac{\Psi_1(\lambda \theta^{(n)}(\lambda))+\eps}{\beta^{(n)}(\lambda)} ,\lambda\right)\le - \frac{\eps}{2n^{1/2}}
    \end{align*} 
    for all $n$ sufficiently large.

    In particular, for $\eps>0$ and uniformly in $\lambda\in [\lambda_0,\lambda_1]$ 
 it holds that \begin{equation}\label{eqn:subset}   
   \dd_n(\lambda)\le  \inf_{s>{n^{-1/2}}} \{H_n(s,\lambda) <  n^{-1/2}\eps/2\}\le \theta^{(n)}(\lambda) + n^{-1/2}\frac{\Psi_1(\lambda\theta^{(n)}(\lambda))}{\beta^{(n)}(\lambda)} + \frac{\eps}n^{-1/2}
    \end{equation}    
    for all $n$ large.
\end{lemma}

\begin{proof} The second statement follows from the first, Proposition \ref{prop:GD} and Lemma \ref{lem:gnlowbound}.
Recall the Taylor expansions near $\theta^{(n)}(\lambda)$ in \eqref{eqn:taylor2} above. Using that, 
\begin{align}
 \nonumber H_n&\bigg(\theta^{(n)}(\lambda) \overbrace{+ \frac{1}{n^{1/2}}\frac{ \Psi(\lambda \theta^{(n)}(\lambda))+\eps}{\beta^{(n)}(\lambda)}}^{x:=},\lambda\bigg)\\
\nonumber &=-(1+o(1))\beta^{(n)}(\lambda)x + n^{-1/2}\Psi_1(\lambda\theta^{(n)}(\lambda) + \lambda x) + \frac{o(1)}{n^{1/2}}
 \\
\nonumber  &= -\frac{\eps(1+o(1))}{n^{1/2}} - n^{-1/2} \Psi_1(\lambda\theta^{(n)}) (1+o(1)) + \Psi_1(\lambda\theta^{(n)}(\lambda) + \lambda x)  + \frac{o(1)}{n^{1/2}}  \\
\label{eqn:helper1}&= -\frac{\eps}{n^{1/2}}+ \frac{o(1)}{n^{1/2}} + n^{-1/2} \left(\Psi_1(\lambda \theta^{(n)} + \lambda x) - \Psi_1(\lambda \theta^{(n)}(\lambda))\right)\le \frac{\eps}{2n^{1/2}}
\end{align} for all large $n$. In \eqref{eqn:helper1}, we used uniform continuity of $\Psi_1$ over compact intervals and the fact that $x\to 0$. The result follows.
\end{proof}

We can now finish the proof of Theorem \ref{thm:DGconv}. 
\begin{proof}[Proof of Theorem \ref{thm:DGconv}]
 We work with the a.s. coupling in \eqref{eqn:Hnexpand}. By Corollary \ref{cor:Gconvto0}, $n^{1/2}\g_n\to 0$ locally uniformly. Using \eqref{eqn:supset} and \eqref{eqn:subset}, for all $\eps>0$ and all $n$ large 
  \begin{equation*}
      \frac{\Psi_1(\lambda\theta^{(n)}(\lambda))}{\beta^{(n)}(\lambda)} - \eps \overset{\eqref{eqn:supset}}{\le} n^{1/2}\left(\dd_n(\lambda)-\theta^{(n)}(\lambda)\right) \overset{\eqref{eqn:subset}}{\le}  \frac{\Psi_1(\lambda\theta^{(n)}(\lambda))}{\beta^{(n)}(\lambda)} +\eps
  \end{equation*} for all $\lambda\in[\lambda_0,\lambda_1]$. The result follows from the uniform convergence of $\theta^{(n)}\to\theta$ and $\beta^{(n)}\to \beta$.
\end{proof}

\end{document}

%% file: preamble.tex
\usepackage[colorlinks= True, citecolor=blue]{hyperref}
\usepackage{amsthm}
\usepackage{amsfonts}
\usepackage{amssymb}
\usepackage{amsmath}
\usepackage{mathrsfs}
\usepackage{mathtools}

\usepackage{fancyhdr}
\usepackage{xcolor}

\usepackage{enumerate}

\newcommand{\g}{\mathfrak{g}}
\newcommand{\dd}{\mathfrak{d}}

\newcommand{\bone}{\mathbf{1}}

\allowdisplaybreaks

\newcommand{\MC}{\operatorname{MC}}

\newtheorem{theorem}{Theorem}[section]
\newtheorem{corollary}[theorem]{Corollary}
\newtheorem{lemma}[theorem]{Lemma}
\newtheorem{proposition}[theorem]{Proposition}

\theoremstyle{definition}

\newtheorem{assumption}[theorem]{Assumption}

\newtheorem{remark}[theorem]{Remark}

\DeclareMathAlphabet{\mathscrbf}{OMS}{mdugm}{b}{n}

\numberwithin{equation}{section}

\newcommand{\dotms}{\dotsm}

\newcommand{\bx}{{\bf x}}

\newcommand{\bw}{{\bf w}}

\newcommand{\eps}{\varepsilon}
\newcommand{\R}{\mathbb{R}}

\newcommand{\D}{{\mathbb{D}}}

\newcommand{\E}{\mathbb{E}}

\newcommand{\bX}{{\bf X}}

\newcommand{\weakarrow}{{\overset{(d)}{\Longrightarrow}}}

\newcommand{\PR}{\mathbb{P}}

\newcommand{\G}{{\mathcal{G}}}

\newcommand{\cC}{{\mathcal{C}}}

\newcommand{\LL}{\mathcal{L}}

\newcommand{\crit}{{\tt crit}}
\newcommand{\Exp}{\textup{Exp}}
\newcommand{\ER}{\textup{ER}}

%% file: arxiv_submission.bbl
\begin{thebibliography}{10}

\bibitem{Aldous.97}
David Aldous, \emph{Brownian excursions, critical random graphs and the multiplicative coalescent}, Ann. Probab. \textbf{25} (1997), no.~2, 812--854. \MR{1434128}

\bibitem{AL.98}
David Aldous and Vlada Limic, \emph{The entrance boundary of the multiplicative coalescent}, Electron. J. Probab. \textbf{3} (1998), no. 3, 59. \MR{1491528}

\bibitem{BN.17}
Frank Ball and Peter Neal, \emph{The asymptotic variance of the giant component of configuration model random graphs}, Ann. Appl. Probab. \textbf{27} (2017), no.~2, 1057--1092. \MR{3655861}

\bibitem{Barbour.19}
AD~Barbour and Adrian R{\"o}llin, \emph{Central limit theorems in the configuration model}, The Annals of Applied Probability \textbf{29} (2019), no.~2, 1046--1069.

\bibitem{BBF.00}
D.~Barraez, S.~Boucheron, and W.~Fernandez de~la Vega, \emph{On the fluctuations of the giant component}, Combin. Probab. Comput. \textbf{9} (2000), no.~4, 287--304. \MR{1786919}

\bibitem{BBS.24}
Shankar Bhamidi, Amarjit Budhiraja, and Akshay Sakanaveeti, \emph{Functional central limit theorems for microscopic and macroscopic functionals of inhomogeneous random graphs}, 2024.

\bibitem{BJR.07}
B\'{e}la Bollob\'{a}s, Svante Janson, and Oliver Riordan, \emph{The phase transition in inhomogeneous random graphs}, Random Structures Algorithms \textbf{31} (2007), no.~1, 3--122. \MR{2337396}

\bibitem{BR.12}
B\'{e}la Bollob\'{a}s and Oliver Riordan, \emph{Asymptotic normality of the size of the giant component via a random walk}, J. Combin. Theory Ser. B \textbf{102} (2012), no.~1, 53--61. \MR{2871766}

\bibitem{CL.02}
Fan Chung and Linyuan Lu, \emph{Connected components in random graphs with given expected degree sequences}, Annals of combinatorics \textbf{6} (2002), no.~2, 125--145.

\bibitem{Corujo.24}
Josu{\'e} Corujo, \emph{The number of connected components in sub-critical random graph processes}, arXiv preprint arXiv:2406.06380 (2024).

\bibitem{CLL.24}
Josu{\'e} {Corujo}, Sophie {Lemaire}, and Vlada {Limic}, \emph{{A novel approach to the giant component fluctuations}}, arXiv e-prints (2024), arXiv:2412.06995.

\bibitem{EFL.23}
Nathana{\"e}l {Enriquez}, Gabriel {Faraud}, and Sophie {Lemaire}, \emph{{The process of fluctuations of the giant component of an Erd{\H{o}}s-R{\'e}nyi graph}}, arXiv e-prints (2023), arXiv:2311.07701.

\bibitem{ER.60}
P.~Erd\H{o}s and A.~R\'{e}nyi, \emph{On the evolution of random graphs}, Magyar Tud. Akad. Mat. Kutat\'{o} Int. K\"{o}zl. \textbf{5} (1960), 17--61. \MR{125031}

\bibitem{EK.86}
Stewart~N. Ethier and Thomas~G. Kurtz, \emph{Markov processes}, Wiley Series in Probability and Mathematical Statistics: Probability and Mathematical Statistics, John Wiley \& Sons, Inc., New York, 1986, Characterization and convergence. \MR{838085}

\bibitem{Janson.20}
Svante Janson, \emph{Asymptotic normality in random graphs with given vertex degrees}, Random Structures \& Algorithms \textbf{56} (2020), no.~4, 1070--1116.

\bibitem{Limic.19}
Vlada Limic, \emph{The eternal multiplicative coalescent encoding via excursions of {L}\'{e}vy-type processes}, Bernoulli \textbf{25} (2019), no.~4A, 2479--2507. \MR{4003555}

\bibitem{MartinLof.86}
Anders Martin-L\"of, \emph{Symmetric sampling procedures, general epidemic processes and their threshold limit theorems}, J. Appl. Probab. \textbf{23} (1986), no.~2, 265--282. \MR{839984}

\bibitem{Neal.03}
Peter Neal, \emph{S{IR} epidemics on a {B}ernoulli random graph}, J. Appl. Probab. \textbf{40} (2003), no.~3, 779--782. \MR{1993267}

\bibitem{Neal.07}
\bysame, \emph{Coupling of two {SIR} epidemic models with variable susceptibilities and infectivities}, J. Appl. Probab. \textbf{44} (2007), no.~1, 41--57. \MR{2312985}

\bibitem{Pittel.90}
Boris Pittel, \emph{On tree census and the giant component in sparse random graphs}, Random Structures Algorithms \textbf{1} (1990), no.~3, 311--342. \MR{1099795}

\bibitem{PW.05}
Boris Pittel and Nicholas~C. Wormald, \emph{Counting connected graphs inside-out}, J. Combin. Theory Ser. B \textbf{93} (2005), no.~2, 127--172. \MR{2117934}

\bibitem{Puhalskii.05}
Anatolii~A. Puhalskii, \emph{Stochastic processes in random graphs}, Ann. Probab. \textbf{33} (2005), no.~1, 337--412. \MR{2118868}

\bibitem{Rath.18}
Bal\'{a}zs R\'{a}th, \emph{A moment-generating formula for {E}rd\"os-{R}\'enyi component sizes}, Electron. Commun. Probab. \textbf{23} (2018), Paper No. 24, 14. \MR{3798235}

\bibitem{Riordan.12}
Oliver Riordan, \emph{The phase transition in the configuration model}, Combin. Probab. Comput. \textbf{21} (2012), no.~1-2, 265--299. \MR{2900063}

\bibitem{Seierstad.13}
Taral~Guldahl Seierstad, \emph{On the normality of giant components}, Random Structures Algorithms \textbf{43} (2013), no.~4, 452--485. \MR{3124692}

\bibitem{Sellke.83}
Thomas Sellke, \emph{On the asymptotic distribution of the size of a stochastic epidemic}, J. Appl. Probab. \textbf{20} (1983), no.~2, 390--394. \MR{698541}

\bibitem{Shorack.79}
Galen~R Shorack, \emph{The weighted empirical process of row independent random variables with arbitrary distribution functions}, Statistica Neerlandica \textbf{33} (1979), no.~4, 169--189.

\bibitem{Startsev.01}
A.~N. Startsev, \emph{Asymptotic analysis of the general stochastic epidemic with variable infectious periods}, J. Appl. Probab. \textbf{38} (2001), no.~1, 18--35. \MR{1816110}

\bibitem{Stepanov.70}
Vadim~E Stepanov, \emph{On the probability of connectedness of a random graph $\mathcal{G}_m(t)$}, Theory of Probability \& Its Applications \textbf{15} (1970), no.~1, 55--67.

\bibitem{Whitt.02}
Ward Whitt, \emph{Stochastic-process limits}, Springer Series in Operations Research, Springer-Verlag, New York, 2002, An introduction to stochastic-process limits and their application to queues. \MR{1876437}

\end{thebibliography}
